\documentclass[10pt]{article}
\font\smallit=cmti10

\usepackage{color}
\usepackage{hyperref}
\usepackage{color}
\usepackage{breakurl}
\usepackage{comment}
\newcommand{\bburl}[1]{\textcolor{blue}{\url{#1}}}

\usepackage{amssymb,latexsym,amsmath,epsfig,amsthm} %% Add other packages as necessary

\makeatletter

\renewcommand\section{\@startsection {section}{1}{\z@}
{-30pt \@plus -1ex \@minus -.2ex}
{2.3ex \@plus.2ex}
{\normalfont\normalsize\bfseries\boldmath}}

\renewcommand\subsection{\@startsection{subsection}{2}{\z@}
{-3.25ex\@plus -1ex \@minus -.2ex}
{1.5ex \@plus .2ex}
{\normalfont\normalsize\bfseries\boldmath}}

\renewcommand{\@seccntformat}[1]{\csname the#1\endcsname. }

\makeatother

\newtheorem{thm}{Theorem}[section]

\newtheorem{cor}[thm]{Corollary}
\newtheorem{lem}[thm]{Lemma}

\newtheorem{rek}[thm]{Remark}

\begin{document}

\begin{center}
\uppercase{\bf When the large divisors of a natural number are in arithmetic progression}
\vskip 20pt
{\bf H\`ung Vi\d{\^e}t Chu}\\

\end{center}
\vskip 20pt

\centerline{\bf Abstract}
\noindent Iannucci considered the positive divisors of a natural number $n$ that do not exceed the square root of $n$ and found all numbers whose such divisors are in arithmetic progression. Continuing the work, we define \textit{large divisors} to be divisors at least $\sqrt{n}$ and find all numbers whose large divisors are in arithmetic progression. The asymptotic formula for the count of these numbers up to a bound $x$ is observed to be $\frac{x\log\log x}{\log x}$.

\noindent 2010 {\it Mathematics Subject Classification}: 11B25

\noindent \emph{Keywords: } divisor, arithmetic progression.
\pagestyle{myheadings} 
\thispagestyle{empty} 
\baselineskip=12.875pt 
\vskip 30pt
\section{Introduction}
For a natural number $n$, let $L_n$ denote the set of positive divisors of $n$ that are at least $\sqrt{n}$ and strictly smaller than $n$; that is,
$$L_n\ :=\ \{d\,:\, d|n, \sqrt{n}\le d<n\}.$$
Also, define $$L'_n\ :=\ \{d\,:\, d|n, \sqrt{n}\le d\le n\}.$$ We call $L'_n$ the set of \textit{large divisors} of $n$. Clearly, $L'_n = L_n + 1$. In this paper, we will determine the set of all natural numbers $n$ such that either $L_n$ or $L'_n$ forms an arithmetic progression. Since $L_n\subset L_n'$, if $L_n'$ forms an arithmetic progression, then so does $L_n$. Hence, we will first focus our attention on $L_n$ and find all $n$ such that 
$$L_n\ =\ \{d, d+ a, d+2a,\ldots, d+(k-1)a\}$$
for some natural numbers $d, a$ and $k$. Note that $L_n$ can be empty and in that case, $L_n$ vacuously forms an arithmetic progression. Let $|L_n| = k\ge 0$.

Our work is a companion to a paper of Iannucci \cite{Ian}, who defined \textit{small divisors} of $n$ to be divisors not exceeding $\sqrt{n}$ and found all natural numbers whose small divisors are in arithmetic progression. For previous work on divisors in or not in arithmetic progression, see \cite{BFL, Var} and on small divisors, see \cite{BTK, Ian2}.

As usual, we have the divisor-counting function
$$\tau(n) \ :=\ \sum_{d|n}1.$$ Since $\tau(n)$ is multiplicative, for the $k$ distinct primes $p_1<p_2<\cdots<p_k$, and natural numbers $a_1, a_2, \ldots, a_k$,
\begin{align}\label{multau}\tau(p_1^{a_1}p_2^{a_2}\cdots p_k^{a_k})\ =\ (a_1+1)(a_2+1)\cdots (a_k+1).\end{align}
If $n = bc$ and $b\le c$, then $b\le \sqrt{n}\le c$; hence
\begin{align}\label{rela}\tau(n) \ =\ \begin{cases} 2|L'_n|, & \mbox{ if } n \mbox{ is not a square;}\\ 2|L'_n|-1,  & \mbox{ if } n \mbox{ is a square} \end{cases}\ =\ \begin{cases}2|L_n|+2, & \mbox{ if } n \mbox{ is not a square;}\\ 2|L_n|+1,  & \mbox{ if } n \mbox{ is a square}.\end{cases}\end{align}
\begin{thm}\label{main}
Let $n$ be a natural number. If numbers in $L_n$ are in arithmetic progression, then one of the following is true: 
\begin{itemize}
    \item [(i)] $n = 1$, hence $L_n = \emptyset$.
    \item [(ii)] $n = p$ for some prime $p$, hence $L_n = \emptyset$.
    \item [(iii)] $n = p^2$ for some prime $p$, hence $L_n = \{p\}$.
    \item [(iv)] $n = p^3$ for some prime $p$, hence $L_n = \{p^2\}$.
    \item [(v)] $n = pq$ for some primes $p<q$, hence $L_n = \{q\}$.
    \item [(vi)] $n = p^4$ for some prime $p$, hence $L_n = \{p^2, p^3\}$.
    \item [(vii)] $n = p^5$ for some prime $p$, hence $L_n = \{p^3, p^4\}$.
    \item [(viii)] $n = p^2q$ for some primes $p<q$, hence $L_n = \{p^2, pq\}$ or $L_n = \{q, pq\}$.
    \item [(ix)] $n = pq^2$ for some primes $p<q$, hence $L_n = \{pq, q^2\}$.
    \item [(x)] $n = pqr$ for some primes $p<q<r$, $pq<r$ and $p=\frac{1}{2}(q+1)$, hence $L_n = \{r, rp, rq\}$.
    \item [(xi)] $n = p^3q$ for some primes $p>q$ and $q = \frac{1}{2}(p+1)$, hence $L_n = \{p^2, p^2q, p^3\}$.
\end{itemize}
\end{thm}

To prove Theorem \ref{main}, we first find all forms of $n$ when $k\le 3$ by case analysis then show that $k$ cannot be larger than $3$. To find all $n$ such that $L'_n$ forms an arithmetic progression, we need only to check the 11 forms in Theorem \ref{main}. It is easy to prove the following corollary, so we omit the proof. 
\begin{cor}\label{largeN}
Let $n$ be a natural number. If numbers in $L_n'$ are in arithmetic progression, then one of the following is true:
\begin{itemize}
    \item [(i)] $n = 1$, hence $L'_n = \{1\}$.
    \item [(ii)] $n = p$, hence $L'_n = \{p\}$.
    \item [(iii)] $n = p^2$ for some prime $p$, hence $L'_n = \{p, p^2\}$.
    \item [(iv)] $n = p^3$ for some prime $p$, hence $L'_n = \{p^2, p^3\}$.
    \item [(v)] $n = pq$ for some primes $p<q$, hence $L'_n=\{q, pq\}$.
\end{itemize}
\end{cor}
%%%%%%%%%%%%%%%%%%%%%%%%%%%%%%%%%%%%%%%%%%%%%%%%%%%%%%%%%%%%%%%%%%%%%%%%%%%%%%%%%%%%%%%%%%%%%%%%%%%%%%%%%%%%%%%%%%%%%%%%%%%%%%%%%%%%%%%%%%%%%%%%%%%%%%%%%%%%%%%%%%%%%%%%%%%%%%%%%%%%%%%%%%%%%%%%%%%%%%%%%%%%%%%%%
%%%%%%%%%%%%%%%%%%%%%%%%%%%%%%%%%%%%%%%%%%%%%%%%%%%%%%%%%%%%%%%%%%%%%%%%%%%%%%%%%%%%%%%%%%%%%%%%%%%%%%%%
\section{Small cases of $|L_n|$}
\begin{lem}\label{allforms}
If $L_n$ forms an arithmetic progression and $k\le 3$, then one of the items in Theorem \ref{main} is true. 
\end{lem}

\begin{proof}

\noindent \textbf{Case 1:} If $k = 0$, then by \eqref{rela}, $\tau(n)\in \{1,2\}$. If $\tau(n) = 1$, then $n = 1$. If $\tau(n) = 2$, then $n = p$ for some prime $p$. Hence, $L_n = \emptyset$. This corresponds to items $(i)$ and $(ii)$ of the theorem.  

\noindent \textbf{Case 2:} If $k = 1$, then by \eqref{rela}, $\tau(n)\in\{3,4\}$. 

If $\tau(n) = 3$, then by \eqref{multau}, $n = p^2$ for some prime $p$, hence $L_n = \{p\}$. This corresponds to item $(iii)$ of the theorem. 

If $\tau(n) = 4$, then by \eqref{multau}, $n = p^3$ for some prime $p$ or $n = pq$ for some primes $p<q$. For the former, $L_n = \{p^2\}$ and for the latter, $L_n = \{q\}$, corresponding to items $(iv)$ and $(v)$ of the theorem. 

\noindent \textbf{Case 3:} If $k = 2$, then by \eqref{rela}, $\tau(n)\in \{5, 6\}$. 

If $\tau(n) = 5$, then by \eqref{multau}, $n = p^4$ for some prime $p$, hence $L_n = \{p^2, p^3\}$. This corresponds to item $(vi)$. 

If $\tau(n) = 6$, then by \eqref{multau}, $n = p^5$ for some prime $p$ or $n = p^2q$ or $pq^2$ for some primes $p<q$.

\begin{itemize}
    \item [] If $n = p^5$, $L_n = \{p^3, p^4\}$.
    \item [] If $n = p^2q$ for some primes $p<q<p^2$, $L_n = \{p^2, pq\}$. If $n = p^2q$ for some primes $p^2<q$, $L_n = \{q, pq\}$.
    \item [] If $n = pq^2$ for some primes $p<q$, $L_n = \{pq. q^2\}$.
\end{itemize}
These correspond to items $(vii), (viii), (ix)$. 

\noindent \textbf{Case 4:} If $k = 3$, then by \eqref{rela}, $\tau(n)\in \{7,8\}$.

If $\tau(n) = 7$, then by $\eqref{multau}$, $n = p^6$ for some prime $p$. Then $L_n = \{p^3, p^4, p^5\}$, which is impossible since $p^5-p^4\neq p^4-p^3$. 

If $\tau(n) = 8$, then by $\eqref{multau}$, $n=pqr$ for some distinct primes $p,q,r$ or $p^3q$ for some distinct primes $p, q$.

\begin{itemize}
    \item [] If $n = pqr$, we may assume that $p<q<r$. Two subcases apply: either $r>pq$ or $r<pq$.
    \begin{itemize}
        \item [] $r> pq$: We have $L_n = \{r, pr, qr\}$ and so, $qr - pr = pr-r$, which implies that $p = \frac{1}{2}(q+1)$. This is item $(x)$.
        \item [] $r<pq$: We have $L_n = \{pq, pr, qr\}$ and so, $qr - pr = pr - pq$, which implies that $p = \frac{qr}{2r-q}$. So, either $(2r-q)| q$ or $(2r-q)|r$. However, both are impossible since $2r-q>r>q$. 
    \end{itemize}
    \item [] If $n = p^3q$, two subcases apply: either $p<q$ or $p>q$.
    \begin{itemize}
        \item [] $p<q$: If $p<q<p^3$, $L_n = \{p^3, pq, p^2q\}$. Either $p^2q - pq = pq - p^3$ or $p^2q-p^3 = p^3 - pq$. It is easy to see that both cases are impossible. If $q>p^3$, $L_n = \{q, pq, p^2q\}$. Since $p^2q - pq = pq - q$, we get $p^2=2p-1$, which implies that $p = 1$, a contradiction. 
        \item [] $p>q$: $L_n = \{p^2,  p^2q, p^3\}$. So, $p^3 - p^2q = p^2q-p^2$. Then $q = \frac{1}{2}(p+1)$. This is item $(xi)$.
    \end{itemize}
\end{itemize}
\end{proof}
%%%%%%%%%%%%%%%%%%%%%%%%%%%%%%%%%%%%%%%%%%%%%%%%%%%%%%%%%%%%%%%%%%%%%%%%%%%%%%%%%%%%%%%%%%%%%%%%%%%%%%%%%%%%%%%%%%%%%%%%%%%%%%%%%%%%%%%%%%%%%%%%%%%%%%%%%%%%%%%%%%%%%%%%%%%%%%%%%%%%%%%%%%%%%%%%%%%%%%%%%%%%%%%%%%%%%%%%%%%%%%%%%%%%%%%%%%%%%%%%%%%%%%%%%%%%%%%%%%%%%%%%%%%%%%%%%%%%%%%%%%%%%%%%%%%%%%%%%%%%%%%%%%%%%%%%%%%%%%
\begin{lem}\label{not4}
Our set $L_n$ cannot have exactly $4$ elements.
\end{lem}
\begin{proof}
We prove by contradiction. Suppose that $|L_n| = 4$. By \eqref{rela}, $\tau(n) \in \{9,10\}$.
\begin{itemize}
    \item[] If $\tau(n) =9$, by \eqref{multau}, $n = p^8$ for some prime $p$ or $n = p^2q^2$ for some primes $p<q$.
    \begin{itemize}
        \item [] If $n = p^8$, $L_n = \{p^4, p^5, p^6, p^7\}$, which cannot form an arithmetic progression.
        \item [] If $n = p^2q^2$ for $p<q$, $L_n = \{pq, p^2q, q^2, pq^2\}$. So, $pq^2 + pq = p^2q + q^2$, which implies that $p=q$, a contradiction.
        \end{itemize}
    \item[] If $\tau(n) = 10$, either $n = p^{9}$ for some prime $p$ or $n = pq^4$ for distinct primes $p,q$.
    \begin{itemize}
        \item [] If $n = p^9$, $L_n = \{p^5, p^6, p^7, p^8\}$, which cannot form an arithmetic progression.
        \item [] If $n = pq^4$, we have four subcases.
        \begin{itemize}
            \item [] $p<q$: $L_n = \{pq^2, q^3, pq^3, q^4\}$, so $pq^2 + q^4 = q^3 + pq^3$, which implies that $p=q$, a contradiction.
            \item [] $q<p<q^2$: $L_n = \{q^3, pq^2, q^4, pq^3\}$, so $q^3 + pq^3 = pq^2+q^4$, which implies that $p =q$, a contradiction. 
            \item [] $q^2<p<q^4$: $L_n = \{pq, q^4, pq^2, pq^3\}$. Either $pq^3+pq=pq^2+q^4$ or $pq^3+q^4=pq+pq^2$. The former gives $q=1$, while the latter gives $p = -\frac{q^3}{q^2-q-1}$. Both pose a contradiction. 
            \item [] $q^4<p$: $L_n = \{p, pq, pq^2, pq^3\}$, so $p+pq^3 = pq+pq^2$, which implies that $q=1$, a contradiction. 
        \end{itemize}
    \end{itemize}
\end{itemize}
Therefore, $|L_n| \neq 4$.
\end{proof}

%%%%%%%%%%%%%%%%%%%%%%%%%%%%%%%%%%%%%%%%%%%%%%%%%%%%%%%%%%%%%%%%%%%%%%%%%%%%%%%%%%%%%%%%%%%%%%%%%%%%%%%%%%%%%%%%%%%%%%%%%%%%%%%%%%%%%%%%%%%%%%%%%%%%%%%%%%%%%%%%%%%%%%%%%%%%%%%%%%%%%%%%%%%%%%%%%%%%%%%%%%%%%%%%%%%%%%%%%%%%%%%%%%%%%%%%%%%%%%%%%%%%%%%%%%%%%%%%%%%%%%%%%%%%%%%%%%%%%%%%%%%%%%%%%%%%%%%%%%%%%%%%%%%%%%%%%%%%%%
\section{Proof of Theorem \ref{main}}
By Lemmata \ref{allforms} and \ref{not4}, to prove Theorem \ref{main}, it suffices to prove that $|L_n|\le 4$. 
\begin{proof}[Proof of Theorem \ref{main}]
We prove by contradiction. Suppose that $k = |L_n|\ge 5$. Recall that
$$L_n \ =\ \{d, d + a, d+2a, \ldots, d+(k-1)a\}$$
for some natural numbers $d$ and $a$. Let $\gcd(d,a) = \ell$. Write $d = \ell k_1$ and $a = \ell k_2$. Clearly, $\gcd(k_1, k_2) = 1$, so there exist integers $s, t$ such that $sk_1 + tk_2 = 1$. 

Let $$M \ =\ [d, d + a, d+2a, \ldots, d+(k-1)a];$$
that is, $M$ denotes the least common multiple of all numbers in $L_n$. Write 
\begin{align*}M &\ =\ [\ell k_1, \ell k_1+ \ell k_2, \ell k_1 + 2\ell k_2, \ldots, \ell k_1 + (k-1)\ell k_2]\\
&\ =\ \ell [k_1, k_1+k_2, k_1+2k_2, \ldots, k_1+(k-1)k_2].
\end{align*}
We claim that $\gcd(k_1+(k-2)k_2, k_1+(k-1)k_2) = 1$. Indeed, let $x = k_1+(k-1)k_2$ and $y = k_1+(k-2)k_2$. We have $x-y = k_2$ and $k_1 = x - (k-1)k_2 = x - (k-1)(x-y)$. Because $sk_1+tk_2 = 1$,
$s(x-(k-1)(x-y))+t(x-y) = 1$. So, $(t+s-s(k-1))x+((k-1)s-t)y=1$, which implies that $\gcd(x,y) = 1$.
Hence, $N = \ell(k_1+(k-2)k_2)(k_1+(k-1)k_2) = \ell[k_1+(k-2)k_2,k_1+(k-1)k_2]$ divides $M$.

Because $N$ divides $M$ and $M$ divides $n$, $N$ divides $n$. Clearly, $N>\ell(k_1+(k-1)k_2) = d+(k-1)a \ge \sqrt{n}$. Because $N\notin L_n$, $N=n$. So, $\ell(k_1+(k-3)k_2)$ divides $N$. Hence, 
$$k_1+(k-3)k_2\mbox{ divides }(k_1+(k-2)k_2)(k_1+(k-1)k_2).$$
Using the same argument as above, we know that $\gcd(k_1+(k-3)k_2,k_1+(k-2)k_2) = 1$. So, 
$$k_1+(k-3)k_2\mbox{ divides }k_1+(k-1)k_2.$$
Write $k_1+(k-1)k_2 = u(k_1+(k-3)k_2)$ for some integer $u\ge 2$. Simplifying the equation, we get
$$\frac{3u-1}{u-1} \ =\ \frac{k_1+kk_2}{k_2}\ =\ \frac{k_1}{k_2} + k \ > \ 5.$$
So, $u<2$. This contradicts that $u\ge 2$. Therefore, $|L_n|<5$, as desired. 
\end{proof}
\begin{rek}\normalfont
We can find out how often a natural number $n$ up to a bound $x>0$ has its large divisors form an arithmetic progression. Let $f(x)$ be the function counting such numbers up to $x$.

The number of $n$ up to $x$ that is either of form $p, p^2$, or $p^3$ for a prime $p$ is asymptotic to
$$\sum_{i=1}^3 \pi(x^{1/i})\ \sim\ \sum_{i=1}^3 \frac{ix^{1/i}}{\log x}.$$
By a result of Landau \cite[\textsection 56]{Lan}, the number of $n\le x$ of the form $pq$ for primes $p<q$ is asymptotic to $$\frac{x\log\log x}{\log x}.$$ Combined with Corollary \ref{largeN}, we know that 
$$f(x)\ \sim\ \frac{x\log\log x}{\log x},$$
which is similar to the asymptotic formula for the case of small divisors \cite{Ian}.
\end{rek} 
%%%%%%%%%%%%%%%%%%%%%%%%%%%%%%%%%%%%%%%%%%%%%%%%%%%%%%%%%%%%%%%%%%%%%%%%%%%%%%%%%%%%%%%%%%%%%%%%%%%%%%%%%%%%%%%%%%%%%%%%%%%%%%%%%%%%%%%%%%%%%%%%%%%%%%%%%%%%%%%%%%%%%%%%%%%%%%%%%%%%%%%%%%%%%%%%%%%%%%%%%%%%%%%%%%%%%%%%%%%%%%%%%%%%%%%%%%%%%%%%%%%%%%%%%%%%%%%%%%%%%%%%%%%%%%%%%%%%%%%%%%%%%%%%%%%%%%%%%%%%%%%%%%%%%%%%%%%

%%%%%%%%%%%%%%%%%%%%%%%%%%%%%%%%%%%%%%%%%%%%%%%%%%%%%%%%%%%%%%%%%%%%%%%%%%%%%%%%%%%%%%%%%%%%%%%%%%%%%%%%%%%%%%%%%%%%%%%%%%%%%%%%%%%%%%%%%%%%%%%%%%%%%%%%%%%%%%%%%%%%%%%%%%%%%%%%%%%%%%%%%%%%%%%%%%%%%%%%%%%%%%%%%%%%%%%%%%%%%%%%%%%%%%%%%%%%%%%%%%%%%%%%%%%%%%%%%%%%%%%%%%%%%%%%%%%%%%%%%%%%%%%%%%%%%%%%%%%%%%%%%%%%%%%%

\vskip 20pt
{\smallit Department of Mathematics, University of Illinois at Urbana-Champaign, Urbana, IL 61820, USA}\\
{\tt hungchu2@illinois.edu}\\ 

\end{document}